\documentclass[11pt]{amsart}
\usepackage{fullpage, amssymb}

\newtheorem{theorem}{Theorem}

\newtheorem{lemma}[theorem]{Lemma}
\newtheorem{corollary}[theorem]{Corollary}

\theoremstyle{definition}

\theoremstyle{remark}
\newtheorem{remark}[theorem]{Remark}

\numberwithin{equation}{section}
\numberwithin{theorem}{section}

\newcommand{\A}{\mathcal{A}}
\newcommand{\B}{\mathcal{B}}

\newcommand{\I}{\mathcal{I}}

\newcommand{\N}{\mathcal{N}}

\newcommand{\z}{\mathcal{Z}}

\begin{document}

\title[Conditional expectations]{Conditional expectations onto\\maximal abelian *-subalgebras}

\author{Charles A. Akemann}
\address{Department of Mathematics\\ University of California \\ Santa Barbara, CA 93106,
USA}
\email{akemann@math.ucsb.edu}

\author{David Sherman}
\address{Department of Mathematics\\ University of Virginia\\ P.O. Box 400137\\ Charlottesville, VA 22904, USA}
\email{dsherman@virginia.edu}


\begin{abstract}
We determine when there is a unique conditional expectation from a semifinite von Neumann algebra onto a singly-generated maximal abelian *-subalgebra.  Our work extends the results of Kadison and Singer via new methods, notably the observation that a unique conditional expectation onto a singly-generated maximal abelian *-subalgebra must be normal.
\end{abstract}

\maketitle

\section{Introduction}

Throughout this paper $\N$ is a von Neumann algebra and $\A \subseteq \N$ is a maximal abelian *-subalgebra (MASA).  Recall that a \textit{conditional expectation} (CE) from a von Neumann algebra onto a subalgebra is a (not necessarily normal) projection of norm one.  Since abelian von Neumann algebras are injective Banach spaces, there is at least one CE from $\N$ onto $\A$.  Here we ask, ``When is there a \textit{unique} CE from $\N$ onto $\A$?"

A MASA is said to be \textit{discrete} if it is generated by minimal projections, and \textit{continuous} if it contains no minimal projections.  Kadison and Singer (\cite {KS}) showed that in $\B(\ell^2)$, a MASA has a unique CE if and only if it is discrete.  A key step in their proof is a calculation in Fourier analysis that guarantees the existence of multiple CEs onto a continuous MASA.  One of the main results here is that for singly-generated $\A$ and semifinite $\N$, the CE is unique if and only if $\A$ has the form $\sum p_t \N p_t$ for a family of abelian projections $\{p_t\} \subset \N$ adding to 1 (Theorem \ref{T:main}).  In particular $\N$ must be of type I.  Interestingly, our proof of this generalization requires no Fourier analysis at all.  Our techniques rely on the new observation that a unique CE onto a singly-generated MASA is necessarily normal (Corollary \ref{T:normal}), and are closely tied to state extensions.

The thrust of the Kadison-Singer paper is to decide whether pure states on a MASA in $\B(\ell^2)$ have unique state extensions to all of $\B(\ell^2)$.  They answered this negatively for a continuous MASA, via the observation that the existence of multiple CEs implies that there is a pure state with multiple extensions.  The converse of this observation is not known to hold, so that the uniqueness of the CE from $\B(\ell^2)$ to a discrete MASA does not entail that pure states have unique state extensions -- a question that remains open as the \textit{Kadison-Singer problem}.  But the observation remains valid for any inclusion of von Neumann algebras, and the results of this paper do answer a Kadison-Singer-type question for many inclusions $\A \subseteq \N$ by guaranteeing that some pure states on $\A$ have nonunique state extensions to $\N$ (Corollary \ref{T:newks}).

We are indebted to Sorin Popa for getting us started on these problems and for suggesting nicer proofs for some of the results.

\section{Background} \label{S:back}

Normality and singularity play an important role in this paper and can be defined in different ways, so we review the characterizations we use.  A linear functional or CE is \textit{normal} if it is weak* continuous.  For a CE $E$, normality is easily seen to be equivalent to the inclusion $E^*(\A_*) \subseteq \N_*$.  A state or CE on $\N$ is \textit{singular} if for any nonzero projection $p \in \N$ there is a nonzero projection $q \leq p$ in its kernel.  A linear functional is singular if it is a linear combination of singular states.

The module actions of $\N$ on its dual will be written as follows:
$$(\varphi x)(y) \triangleq \varphi(xy), \qquad (x\varphi)(y) \triangleq \varphi(yx), \qquad x,y \in \N, \: \varphi \in \N^*.$$
Here the normality or singularity of $\varphi$ implies the same property for $\varphi x$ and $x \varphi$.  The \textit{centralizer} of $\varphi \in \N^*_+$ is the *-subalgebra $\{x \in \N \mid\varphi x = x \varphi\}$.  

Let $W$ be the collection of finite sets of projections in $\A$ with sum 1, partially ordered by refinement; i.e., $F \geq G$ if every element of $F$ is dominated by an element of $G$.   For $F \in W, x \in \N, \varphi \in \N^*$, define the ``pavings" $x_F \in N$ and $\varphi_F\in \N^*$ by  $x_F = \sum_{p \in F}pxp$ and $g_F = \sum_{p \in F} p\varphi p$.  With $E: \N \to \A$ a CE, the following facts are easy to check:
\begin{equation} \label{E:paving}
\varphi_F(x)=\varphi(x_F), \qquad \|x_F\| \leq\| x\|, \qquad \|\varphi_F\| \leq \|\varphi\|, \qquad E^*(\varphi|_\A)_F = E^*(\varphi|_\A).
\end{equation}


If $V$ is an upward-filtering subset of $W$ such that $(\cup_{F \in V} F)'\cap \N =\A$, we call $V$ a \textbf{full subset} for $\A$.  In the sequel it will be useful to work with sequential full subsets; these clearly exist for any singly-generated $\A$, since we may take an increasing family in a countable set of projections that generates $\A$.  But full subsets need not be generating, and large MASAs may also have sequential full subsets.  For example, let $G$ be a group whose generators $\{g_t\}_{t \in [0,1]}$ satisfy only the relations $g_s g_t = g_t g_s$ for $s,t > 0$.  Since $G$ is ICC, $L(G)$ is a $\text{II}_1$ factor.  Moreover $\A = W^*(\{g_t\}_{t>0})$ is an uncountably-generated MASA, and $\A = \{g_1\}' \cap L(G)$.  (This can be read, for instance, out of \cite[Proposition 4.1]{Popa}.)  Thus any upward-filtering $V$ whose union generates $W^*(\{g_1\})$ will be full.

This gives a way to produce a CE $E: \N \to \A$.  Let $V$ be a full subset for $\A$, and consider the net $\{(x_F)_{x \in \N} \}_{F \in V}$.  For each index $F$ the output lies in $\Pi_{x \in \N} \N_{\|x\|}$ (here $\N_{\|x\|}$ denotes the closed ball in $\N$ of radius $\|x\|$), which is compact when topologized as the product of weak* compact sets.  Let $U$ be the index set for a convergent subnet, and finally set $E(x) = w^*\lim_{F \in U} x_{F}$.  The basic idea of this construction originates with von Neumann (\cite[Chapter II]{vN}).  It was explicitly studied by Kadison-Singer for sequential full subsets of MASAs in $\B(\ell^2)$, and their justification that $E$ is a CE holds in the general situation (\cite[Lemma 1]{KS}).  We follow their nomenclature by calling a CE onto a MASA \textit{proper} if it is of this form, and otherwise \textit{improper}.  (This terminology has been applied slightly differently by some later authors.)  The CEs constructed by Kadison-Singer from $\B(\ell^2)$ onto a continuous MASA are all proper -- are there any improper CEs for this inclusion?  We do not know.  In fact we believe that this paper is the first to establish that improper CEs onto (other) MASAs exist (Corollary \ref{T:improper}).

From the extensive literature concerning CEs onto MASAs, here are the theorems that we need.

\begin{theorem} \label{T:facts} ${}$
\begin{enumerate}
\item A CE is a positive bimodule map: $E(a_1 x a_2) = a_1 E(x) a_2$ for $x \in \N$, $a_j \in \A$ $($\cite[Theorem 1]{To1957}$)$.
\item A normal CE onto a MASA is the unique proper CE $($\cite[Corollary 6.1.8]{Ar}$)$.  Thus there can be at most one normal CE onto a MASA.
\item A normal CE onto a MASA is automatically faithful $($\cite[Proposition 1.2]{To}$)$.  The existence of a faithful CE implies the existence of a normal CE $($\cite[Proposition 2.2]{To}$)$, and for a MASA in a semifinite algebra this happens if and only if the MASA is generated by finite projections $($\cite[Proposition 4.4]{To}$)$.
\item There need not be any normal CEs -- for instance, when $\A$ is the continuous MASA in $\B(\ell^2)$ $($\cite[Remark 5]{KS}$)$.
\item Let $\psi$ be a normal faithful state on $\A$.  There is a 1-1 correspondence between CEs from $\N$ to $\A$ and state extensions of $\psi$ with $\A$ in their centralizers, given by $E \leftrightarrow E^*(\psi)$.  Moreover $E$ is normal or singular if and only if $E^*(\psi)$ is.
\end{enumerate}
\end{theorem}

\begin{proof}
We discuss only (5), for which we have no complete reference.

First, any state of the form $E^*(\psi)$ has $\A$ in its centralizer:
$$(a E^*(\psi))(x) = \psi(E(xa)) = \psi(E(x)a) = \psi(a E(x)) = \psi(E(ax)) = (E^*(\psi)a)(x), \qquad a \in \A, \: x \in \N.$$
It is shown in \cite[Theorem 1]{dK} that any state extension of $\psi$ having $\A$ in its centralizer can be written as $E^*(\psi)$ for some $E$.  As for uniqueness, suppose $E,E': \N \to \A$ are CEs satisfying $E^*(\psi) = E'^*(\psi)$, and choose an arbitrary $x \in \N$.  We have
\begin{equation} \label{E:unique}
(\psi a)(E(x) - E'(x)) = \psi (E(ax) - E'(ax)) = E^*(\psi)(ax) - E'^*(\psi)(ax) = 0, \qquad a \in \A.
\end{equation}
Since $\psi$ is faithful, the space $\psi \A$ is norm dense in $\A_*$ (by \cite[Theorem III.2.7(iii)]{T}, for instance).  So \eqref{E:unique} implies that all normal functionals vanish on $E(x)-E'(x)$, and therefore $E(x)=E'(x)$.

If $E$ is normal, $E^*(\psi) = \psi \circ E$ is normal as a composition of normal maps.  On the other hand, if $E^*(\psi)$ is normal, for any $a \in \A$ we have $E^*(\psi a) = E^*(\psi) a \in \N_*$.  Again by density of $\psi \A$ in $\A_*$, we conclude $E^*(\A_*) \subseteq \N_*$, and $E$ is normal.  The statement about singularity follows from the observation that $E$ and $E^*(\psi)$ annihilate the same projections.
\end{proof}

Items (2) and (5) of Theorem \ref{T:facts} entail the well-known fact that for a MASA in a finite von Neumann algebra, the unique CE that preserves normal tracial states is the unique normal CE.

\section{Uniqueness implies normality}

The main result of this section is the implication (1) $\Rightarrow$ (5) in the following theorem.

\begin{theorem}\label{T:unique}
For a CE $E: \N \to \A$, the following conditions are equivalent:
\begin{enumerate}
\item $E$ is the unique proper CE;
\item for every full subset $V$ for $\A$, we have $\forall x \in \N$, $w^*\lim_{F \in V} x_F = E(x)$;
\item for every full subset $V$ for $\A$, we have $\forall \varphi \in \N_* \subseteq \N^*$, $w^*\lim_{F \in V} \varphi_F = E^*(\varphi|_\A)$.
\end{enumerate}
If $\A$ has a sequential full subset $\{F_n\}$, then the following conditions are also equivalent:
\begin{enumerate}
\item[(4)] $\forall \varphi \in \N_*$, $\varphi_{F_n} \to E^*(\varphi|_A)$ in norm;
\item[(5)] $E$ is normal.
\end{enumerate}
\end{theorem}

\begin{proof}


(1) $\Leftrightarrow$ (2): This follows from the definition of a proper CE.  

(2) $\Rightarrow$ (3): From the equality
$$E^*(\varphi|_\A)(x) =\varphi(E(x))=\varphi(w^*\lim x_F)=\lim \varphi(x_F) = \lim \varphi_F(x), \qquad x \in \N, \: \varphi \in \N_*.$$

(3) $\Rightarrow$ (2): Similarly, from
$$\varphi(E(x)) = E^*(\varphi|_\A)(x) = (w^*\lim \varphi_F)(x) = \lim \varphi_F(x) = \lim \varphi(x_F), \qquad x \in \N, \: \varphi \in \N_*.$$

(3) $\Rightarrow$ (4): Condition (3) implies $\N_* \ni \varphi_{F_n} \to E^*(\varphi|_\A)$ in the weak* topology of $\N^*$.  By \cite[Corollary 3.3]{ADG}, weak* convergence of the sequence is equivalent to \textit{weak} convergence.  Since norm closed convex hulls and weakly closed convex hulls agree, there is a sequence of convex combinations of $\{\varphi_{F_n}\}$ that converges in norm to $E^*(\varphi|_\A)$.  Now for any convex combination $\sum_{j=1}^N c_j \varphi_{F_{n_j}}$, for $n \geq \max\{n_j\}$ the facts in \eqref{E:paving} give
$$\left\| \varphi_{F_n}  - E^*(\varphi|_\A) \right\| = \left\| \left(\sum_{j=1}^N c_j \varphi_{F_{n_j}} - E^*(\varphi|_\A)\right)_{F_n}\right\| \leq \left\| \sum_{j=1}^N c_j \varphi_{F_{n_j}} - E^*(\varphi|_\A) \right\|.$$
It follows that $\varphi_{F_n} \to E^*(\varphi|_\A)$ in norm.

(4) $\Rightarrow$ (5): Since $\N_*$ is a norm-closed subspace of $\N^*$, (2) implies $E^*(\varphi|_\A) \in \N_*$.  Any normal state on $\A$ is the restriction of a normal state on $\N$ (\cite[Exercise III.5.1]{T}), so $E^*(\A_*) \subseteq \N_*$, and $E$ is normal.  

(5) $\Rightarrow$ (1): This is Theorem \ref{T:facts}(2).
\end{proof}

\begin{corollary}$($\cite[Theorem 2]{KS}$)$ \label{T:ks}
There is more than one proper CE onto a continuous MASA in $\B(\ell^2)$.
\end{corollary}

\begin{proof}
Immediate from Theorems \ref{T:facts}(4) and \ref{T:unique}.
\end{proof}

As mentioned in the Introduction, Kadison-Singer's original proof of Corollary \ref{T:ks} used Fourier analysis.

\begin{corollary} \label{T:normal}
If there is a unique CE onto a singly-generated MASA, then this CE is normal and faithful.
\end{corollary}

\begin{proof}
Since proper CEs always exist, a unique CE is the unique proper CE.  A singly-generated MASA has a sequential full subset, so the conclusion follows from Theorems \ref{T:unique} and \ref{T:facts}(3).
\end{proof}

As in Section \ref{S:back}, we let $W$ be the net of all finite sets of projections from $\A$.  We will say that an operator $x \in \N$ is \textit{pavable} if there is a sequence $\{F_n\} \subset W$ such that $x_{F_n}$ converges in norm to an element of $\A$.  Kadison-Singer showed that in $\B(\ell^2)$, this is equivalent to requiring that whenever two states of $\N$ restrict to the same pure state of $\A$, they agree on $x$ (\cite[Lemma 5]{KS}).  Their arguments remain valid in our setting.

\begin{theorem} \label{T:paving}
For a CE $E: \N \to \A$, the following conditions are equivalent:
\begin{enumerate}
\item every pure state of $\A$ has a unique state extension to $\N$;
\item every operator in $\N$ is pavable;
\item $\forall x \in \N$, $\lim_{F \in W} x_F= E(x)$ (norm limit);
\item $\forall x \in \N$, $\lim_{F \in W} x_F= E(x)$ (weak limit);
\item $\forall\varphi \in \N^*$, $w^*\lim_{F \in W} \varphi_F=E^*(\varphi |_\A)$;
\item if $\varphi$ is a pure state of $\N$ that restricts to a pure state of $\A$, then $w^*\lim_{F \in W} \varphi_F=E^*(\varphi |_\A)$.
\end{enumerate}
\end{theorem}

\begin{proof} (1) $\Leftrightarrow$ (2): As mentioned just before the theorem, this can be proved in the same way as \cite[Lemma 5]{KS}.

(2) $\Leftrightarrow$ (3): The reverse implication is trivial, so we suppose that $x$ is pavable: there are $a \in \A$ and $\{F_n\} \subset W$ such that $\|x_{F_n} - a \| \to 0$.  Then for any $n$, $\|E(x) - a\| =|E(x_{F_n} - a)\| \leq \|x_{F_n} - a\|$, so $a = E(x)$.  And for any $F \geq F_n$, $\|x_F - E(x)\| = \|(x_{F_n} - E(x))_F\| \leq \|x_{F_n} - E(x)\|$.

(3) $\Rightarrow$ (4): Trivial. 

(4) $\Leftrightarrow$ (5): These are the same computations as (2) $\Leftrightarrow$ (3) in Theorem \ref{T:normal} (with obvious small modifications).

(5) $\Rightarrow$ (6): Trivial.

(6) $\Rightarrow$ (1): Suppose there are two pure states $\varphi_1, \varphi_2$ on $\N$ such that $\varphi_1 |_\A=\varphi_2 |_\A$ is pure.  As observed in the first paragraph of the proof of \cite[Lemma 5]{KS}, for $j=1,2$ and $F \in W$ one has $(\varphi_j)_F = \varphi_j$.  (This is essentially because for a projection in $\A$, $\varphi_j$ annihilates either the projection or its complement.)  Thus the condition in (6) implies
$$\varphi_1 = w^*\lim (\varphi_1)_F = E^*(\varphi_1 |_\A) = E^*(\varphi_2 |_\A) = w^*\lim (\varphi_2)_F = \varphi_2.$$
It is a standard fact that a pure state on $\A$ has a unique state extension if and only if it has a unique pure state extension.  (The pure state extensions are the extreme points of the convex weak* compact set of state extensions.)
\end{proof}

\begin{remark} \label{T:stronger}
Since the first two items in Theorem \ref{T:paving} do not refer to $E$, they clearly imply uniqueness of the CE.  Regarding unique state extensions of pure states, this well-known observation can also be seen more directly and goes back to Kadison-Singer.
\end{remark}

\section{MASAs of semifinite von Neumann algebras} \label{S:CE1}

\begin{theorem} \label{T:typeI}
If $\N$ is type I, the following conditions are equivalent:
\begin{enumerate}
\item there is a normal CE onto $\A$;
\item there exist abelian projections $\{p_t\} \subset \A$ with $\sum p_t = 1$.
\end{enumerate}
These conditions imply
\begin{enumerate}
\item[(3)] there is a unique CE from $\N$ onto $\A$.
\end{enumerate}
If $\A$ has a sequential full subset, then all three conditions are equivalent.
\end{theorem}

\begin{proof}

(1) $\Leftrightarrow$ (2): By Theorem \ref{T:facts}(3) the existence of a normal CE is equivalent to $\A$ being generated by finite projections.  If $q \in \A$ is finite, then $q\A$ is a MASA in the finite type I algebra $q\N q$, so $q$ is a sum of abelian projections (\cite [Exercise 6.9.23]{KR}).

(2) $\Rightarrow$ (3): Let $\{p_t\}$ be abelian projections in $\A$ such that $\sum p_t = 1$.  Note that for any $a \in \A$,
$p_s a p_t = \delta_{st} p_s a p_s$.  Further note that for any $x \in \N$, $p_s x p_s$ belongs to the abelian algebra $p_s \N p_s$, so it commutes with $p_s \A$ and thus all of $\A$.  Since $\A$ is a MASA, $p_s x p_s \in \A$.  Now let $E: \N \to \A$ be any CE and compute
$$E(x) = \left(\sum p_s \right) E(x) \left(\sum p_t \right) = \sum p_s E(x) p_s = \sum E(p_s x p_s) = \sum p_s x p_s.$$
(All sums should of course be interpreted as $\sigma$-strong limits of finite sums.)  Thus the only CE from $\N$ onto $\A$ is $x \mapsto \sum p_s x p_s$, which is visibly normal.

(3) $\Rightarrow$ (1): Assuming the sequential full subset, this follows from Theorem \ref{T:unique}.
\end{proof}




\begin{remark} \label{T:sfs}
If $\N$ is type I, $\A$ has a sequential full subset, and in addition $\N$ has singly-generated center $\z$, then $\A$ must be singly-generated.  For assume these hypotheses, and let $\{F_n\}$ be a sequential full subset for $\A$.  After enlarging $\{F_n\}$ if necessary we may assume that $\z \subseteq W^*(\{F_n\})$.  Now we apply the classical fact that type I algebras are \textit{normal}, meaning that any subalgebra that contains the center is equal to its own double relative commutant (\cite[Exercice III.7.13b]{D}):
$$\A = \A' \cap \N = (W^*(\{F_n\})' \cap \N)' \cap \N = W^*(\{F_n\}).$$
Thus $\A$ is singly-generated (\cite[Lemma III.1.20]{T}).
\end{remark}

For the type II case discussed in the next two results, the main points are these: given a normal tracial state $\tau$ on $\N$, there is a singular state $\varphi$ on $\N$ that agrees with $\tau$ on $\A$; under a cardinality restriction, we can also ensure that $\A$ lies in the centralizer of $\varphi$; by Theorem \ref{T:facts}(5) this produces a singular CE.  Some related arguments can be found in \cite[Proposition 2.4, Corollary 2.5, and Paragraph 4.2]{Pop} and \cite[Lemma 4.2 and subsequent text]{P}.  We thank Sorin Popa for his suggestions on these constructions.

\begin{lemma}\label{T:popa}
If $\tau$ is a normal tracial state on the $\text{II}_1$ von Neumann algebra $\N$, then $\tau|_\A$ extends to a singular state on $\N$.  Actually any normal state on $\A$ extends to a singular state on $\N$.
\end{lemma}




\begin{proof}
After compressing by the support of $\tau$, which is central and thus an element of $\A$, we may assume that $\tau$ is faithful.  By \cite[Exercise 6.9.29]{KR}, for any $n$ there are projections $\{q^n_j\}_{j=1}^{2^n} \subset \A$ that are equivalent in $\N$ and have sum $1$.  For each $1 \leq i,j \leq 2^n$, let $v^n_{ij}$ be a partial isometry effecting the equivalence of $q^n_i$ and $q^n_j$, with the requirements that $v^n_{ji} = (v^n_{ij})^*$ and $v^n_{ii} = q^n_i$.  We set $p_n =2^{-n} \sum_{i,j=1}^{2^n}  v^n_{ij}$, which is easily checked to be a projection.  With $E$ the normal $\tau$-preserving CE onto $\A$, we also compute
$$E(v^n_{ij}) = E(q^n_i v^n_{ij} q^n_j) =  q^n_i E(v^n_{ij}) q^n_j = \delta_{ij}q^n_i \quad \Rightarrow \quad E(p_n) = 2^{-n} \sum_{i,j} E(v^n_{ij}) = 2^{-n} \sum_i q^n_i = 2^{-n}1.$$

Define a sequence of states on $\N$ by $\varphi_n=2^n\tau(\cdot p_n)$.  Note that
$$\varphi_n(a)=2^n\tau(a p_n)=2^n\tau(E(a p_n)) =2^n \tau(aE(p_n)) =2^n \tau(a(2^{-n}1)) = \tau(a), \qquad a \in \A,$$
so that any weak* limit point $\varphi$ of $\{\varphi_n\}$ in $\N^*$ extends $\tau|_\A$.  Moreover, for any $m$,
$$\varphi\left(\bigvee_{k=m}^\infty p_k\right) =(w^*\lim \varphi_n)\left(\bigvee_{k=m}^\infty p_k \right) \geq (w^*\lim \varphi_n)(p_n) = 1, \quad \tau\left(\bigvee_{k=m}^\infty p_k\right) \leq \sum_m^\infty \tau(p_k) = 2^{-m+1}.$$
Considering the complements of the projections $\vee_{k=m}^\infty p_k$, we see that $\varphi$ vanishes on projections of trace arbitrarily close to 1.  Now given any projection $p \in \N$, find another projection $q$ with $\varphi(q)=0$ and $\tau(p)+ \tau(q) >1$.  The formula $p - (p \wedge q) \sim (p \vee q) - q$ implies $\tau(p \wedge q) = \tau(p) + \tau(q) - \tau(p \vee q) > 0$.  Thus $0 \neq p\wedge q \leq p$ and $\varphi(p \wedge q) \leq \varphi(q) = 0$, as required to show that $\varphi$ is singular.

For the second sentence of the lemma, let $\psi$ be a normal state on $\A$.  After compressing by the support of $\psi$, which is $\sigma$-finite, we may assume that $\N$ admits a faithful normal tracial state $\tau$.  By the preceding argument, $\tau|_\A$ has a singular extension $\varphi$.

Since $\tau|_\A \A$ is norm dense in $\A_*$, there are $\{a_n\} \subset \A$ with $\tau|_\A a_n \to \psi$ in norm.  Let $\rho$ be a weak* limit point of the singular functionals $\{\varphi a_n\}$.  Necessarily $\rho$ is singular (\cite[Theorem III.5]{Ak}).  But $(\varphi a_n)|_\A = \tau|_\A a_n$ now converges both weak* to $\rho|_\A$ and in norm to $\psi$, so that $\rho$ must restrict to $\psi$ on $\A$.
\end{proof}

\begin{theorem} \label{T:sg}
If $\N$ is type II and $\A$ is singly-generated, then there are multiple CEs from $\N$ onto $\A$.
\end{theorem}

\begin{proof}
Assume the hypotheses, and suppose toward a contradiction that $E: \N \to \A$ is the unique CE.  By Corollary \ref{T:normal}, $E$ is normal and faithful.  
By Theorem \ref{T:facts}(3) we know that $\A$ contains a nonzero finite projection $r$.  Let $\psi$ be any normal state on $r\A$; after compressing by the support of $\psi$, we may assume that $\psi$ is faithful on $\A$ and that $\N$ is type $\text{II}_1$.  Our strategy is to find a non-normal state extension of $\psi$ that has $\A$ in its centralizer, so that by Theorem \ref{T:facts}(5) there is also a non-normal CE onto $\A$.

Let $\{p_n\}$ be a countable generating set of projections for $\A$.  By Lemma \ref{T:popa} there is a singular $\varphi_1$ that restricts to $\psi$ on $\A$.  We recursively define $\varphi_n=p_n \varphi_{n-1}p_n+(1-p_n)\varphi_{n-1}(1-p_n)$, a singular state that restricts to $\psi$ on $\A$ and contains $\{p_1,p_2,\dots,p_n\}$ in its centralizer.  Let $\varphi$ be a weak* limit point of $\{\varphi_n\}$.  Then $\varphi$ is still singular (\cite[Theorem III.5]{Ak}), still restricts to $\psi$ on $\A$, and has all the $p_n$ in its centralizer.  The proof will be complete if we can show that $\A$ lies in its centralizer, which we do now by adapting the idea of \cite[Proof of Theorem 11]{C}.

Let $\A_0$ be the intersection of $\A$ and the centralizer of $\varphi$.  Then $\A_0$ is a *-algebra, and we claim that it is closed in the $\sigma$-strong* topology.  
For if $\{a_\alpha\} \subset \A_0$ is a net converging $\sigma$-strong* to $a \in \A$, then for any $x \in \N$ we have
\begin{align*}
|(a\varphi-\varphi a)(x)| &= |[(a - a_\alpha)\varphi - \varphi (a - a_\alpha)](x)| \\ &= \lim_\alpha |[(a - a_\alpha)\varphi - \varphi (a - a_\alpha)](x)|\\ &= \lim_\alpha |\varphi(x(a-a_\alpha)) - \varphi((a-a_\alpha)x)| \\ &\leq \lim_\alpha \varphi(xx^*)^{1/2} \varphi((a - a_\alpha)^*(a- a_\alpha))^{1/2} + \varphi((a - a_\alpha)(a- a_\alpha)^*)^{1/2} \varphi(x^*x)^{1/2} \\ &= 0,
\end{align*}
since $\varphi|_\A = \psi$ is normal.  Therefore $\A_0$, being $\sigma$-strong* closed, is a von Neumann subalgebra of $\A$.  We have already noted that $\{p_n\} \subset \A_0$, so $\A_0 = \A$; i.e., $\A$ lies in the centralizer of $\varphi$.
\end{proof}

\begin{corollary} \label{T:improper}
There is an improper CE onto a singly-generated MASA in a type $\text{II}_1$ algebra.
\end{corollary}

\begin{proof}
The normal CE is the unique proper CE, by Theorem \ref{T:facts}(2).
By Theorem \ref{T:sg} there are others.
\end{proof}

To our knowledge there had been no previous examples of improper CEs onto MASAs.

\begin{theorem} \label{T:main}
If $\A$ is singly-generated and $\N$ is semifinite, the following conditions are equivalent:
\begin{enumerate}
\item there exist abelian projections $\{p_t\} \subset \A$ with $\sum p_t = 1$ (and in particular, $\N$ is type I);
\item no normal state of $\A$ has a non-normal state extension to $\N$;
\item no normal state of $\A$ has a singular state extension to $\N$;
\item there is a unique CE $E : \N \to \A$ that is also normal and faithful;
\item there is a unique CE $E : \N \to \A$.
\end{enumerate}
\end{theorem}

\begin{proof}
(1) $\Rightarrow$ (2): Suppose that the $\{p_t\}$ exist, and that $\varphi$ is a state of $\N$ such that $\varphi|_\A$ is normal.  Let $\varphi = \varphi_1 + \varphi_2$ be the unique decomposition in which $\varphi_1$ is normal and $\varphi_2$ is singular (\cite[Theorem III.2.14]{T}); necessarily $\varphi_1$ and $\varphi_2$ are positive.  For any $t$ it follows from the definition of singularity that the restriction $\varphi_2|_{p_t\N p_t}$ is still singular.  On the other hand $p_t$ is an abelian projection and $\A$ is a MASA, so $p_t\N p_t \subseteq \A$ and by our assumption $\varphi_2|_{p_t\N p_t}$ must also be normal.  But a singular normal positive functional is zero, and therefore $\varphi_2(p_t) = 0$.  Now $\varphi_2|_\A = \varphi|_\A - \varphi_1|_\A$ is normal, so $\varphi_2(1) = \varphi_2(\sum p_t) = \sum \varphi_2(p_t) = 0$.  Thus $\varphi_2 = 0$, and $\varphi = \varphi_1$ is normal.


(2) $\Rightarrow$ (3): Trivial.

(3) $\Rightarrow$ (1): Suppose that (1) is false, so that there is a projection $0 \ne p \in \A$ such that $p\A$ contains no abelian projections.   We need to show that there is a normal state on $p\A$ extending to a singular state on $p\N p$, and it suffices to assume $p=1$.  There are two cases.

Case 1:  Suppose $\A$ contains no finite projections.  Let $\I$ be the closed ideal of $\N$ generated by the finite projections.  The dual space of $\N/\I$ is positively isometric to $\I^\perp$, which by weak* density of $\I$ in $\N$ consists entirely of singular linear functionals.  Since $\A$ contains no finite projections, $\A$ is isometrically imbedded in $\N/\I$, hence its dual space is the set of restrictions of functionals in $\I^\perp$.  We conclude that every normal state of $\A$ is the restriction of a singular state of $\N$.

Case 2:  If $\A$ contains a non-zero finite projection $q$, then $q\A$ is a MASA of $q\N q$.  We may assume that $q=1$, so that $\N$ is finite.  Now $\N$ cannot have a type I summand, because again by \cite[Exercise 6.9.23]{KR} $\A$ would have nonzero abelian projections, contrary to assumption.  So $\N$ is type $\text{II}_1$, and the conclusion follows from Lemma \ref{T:popa}.

(1) $\Leftrightarrow$ (5): This follows from Theorems \ref{T:typeI} and \ref{T:sg}.


(4) $\Leftrightarrow$ (5): The nontrivial direction is covered by Corollary \ref{T:normal}.
\end{proof}

The argument for (3) $\Rightarrow$ (1) allows for a more refined conclusion.  Let $z \in \A$ be the supremum of all projections in $\A$ that are abelian in $\N$, and let $\psi$ be a state on $\A$ with support $p$.  Then $p \leq z$ if and only if $\psi$ has a unique state extension to $\N$, necessarily normal;  $p \leq 1-z$ if and only if $\psi$ has a singular state extension.

\begin{corollary} \label{T:newks}
Let $\A$ be a singly-generated MASA in the semifinite von Neumann algebra $\N$.  If $\A$ is not generated by abelian projections (in particular, if $\N$ is not type I), then some pure states of $\A$ have nonunique state extensions to $\N$.
\end{corollary}

\begin{proof}


Under the hypotheses, Theorem \ref{T:main} implies that there are multiple CEs from $\N$ to $\A$.  By Remark \ref{T:stronger} there must be some pure state on $\A$ with multiple state extensions.
\end{proof}

\end{document}